\documentclass[12pt,a4paper,dvipdfmx]{article}
\usepackage{amsmath,amsthm,amssymb,mathrsfs,enumerate,hyperref}
\usepackage[numbers]{natbib}

\theoremstyle{plain}
\newtheorem{thm}{Theorem}[section]
\newtheorem{lem}{Lemma}[section]
\newtheorem{cor}{Corollary}[section]
\newtheorem{prop}{Proposition}[section]

\theoremstyle{definition}
\newtheorem{dfn}{Definition}[section]
\newtheorem{exmp}{Example}[section]
\newtheorem{rem}{Remark}[section]

\def\R{{\mathbb{R}}}
\def\U{{\mathcal{U}}}

\numberwithin{equation}{section}
\allowdisplaybreaks

\title{Decomposability, Convexity and Continuous Linear Operators in $L^1(\mu,E)$: The Case for Saturated Measure Spaces\thanks{This research is supported by JSPS KAKENHI Grant Number JP18K01518 from the Ministry of Education, Culture, Sports, Science and Technology, Japan.}}
\date{\today}
\author{Nobusumi Sagara\thanks{I am benefitted by the useful comment from M. Ali Khan.} \\ \\
{\small Faculty of Economics, Hosei University} \\
{\small 4342, Aihara, Machida, Tokyo, 194-0298, Japan} \\
{\footnotesize e-mail: nsagara@hosei.ac.jp}}

\begin{document}
\maketitle
\setcounter{page}{0}
\thispagestyle{empty}

\begin{abstract} 
Motivated by the Lyapunov convexity theorem in infinite dimensions, we extend the convexity of the integral of a decomposable set to separable Banach spaces under the strengthened notion of nonatomicity of measure spaces, called ``saturation'', and provide a complete characterization of decomposability in terms of saturation. 
\end{abstract}


\clearpage
\tableofcontents

\section{Introduction}
Let $(\Omega,\Sigma,\mu)$ be a finite measure space. Denote by $L^1(\mu,\R^n)$ the space of integrable functions from $\Omega$ to $\R^n$, normed by $\| f \|_1=\int \| f(\omega) \|d\mu$ for $f\in L^1(\mu,\R^n)$. A subset $K$ of $L^1(\mu,\R^n)$ is said to be \textit{decomposable} if $\chi_Af+(1-\chi_A)g\in K$ for every $f,g\in K$ and $A\in \Sigma$, where $\chi_A$ is the characteristic function of $A$. The notion of decomposability was originally introduced in \cite{ro68,ro76} to explore the duality theory for integral functionals defined on decomposable sets, which had turned out to be useful in optimal control, variational geometry, differential inclusions, and fixed point theorems in the context of nonconvexity; see \cite{cr98,fr04,mw19,ol90}. 

Our main concern in this paper is an attempt to extend the following specific result on decomposability to infinite dimensions. 

\begin{thm}[\citet{ol84}]
\label{thm1}
Let $(\Omega,\Sigma,\mu)$ be a nonatomic finite measure space. If $K$ is a decomposable subset of $L^1(\mu,\R^n)$, then the set $\{ \int fd\mu\mid f\in K \}$ is convex. Furthermore, if $K$ is bounded and closed, then $\{ \int fd\mu\mid f\in K \}$ is compact and convex. 
\end{thm}

\noindent
In Theorem \ref{thm1}, decomposability substitutes the convexity of $K$ because the latter implies decomposability and the convexity of $\{ \int fd\mu\mid f\in K \}$, which reveals the fact that nonatomicity together with decomposability replaces the consequence of the classical Lyapunov convexity theorem that guarantees the convexity of the integral of a multifunction. As shown below, Theorem \ref{thm1} is, however, no longer true for every infinite dimensional Banach space in view of the celebrated failure of the Lyapunov convexity theorem in infinite dimensions. It is well-known that the Lyapunov convexity theorem is valid only in the sense of approximation in that the closure operation is involved; see \cite{uh69} and Remark \ref{rem1} below. 

Motivated by the Lyapunov convexity theorem in infinite dimensions established in \cite{ks13}, we extend the convexity of the integral of a decomposable set to separable Banach spaces under the strengthened notion of nonatomicity of measure spaces, called ``saturation'', and provide a complete characterization of decomposability in terms of saturation.

\section{Preliminaries}
\subsection{Bochner Integrals in Banach Spaces}
\label{subsec1}
Let $(\Omega,\Sigma,\mu)$ be a complete finite measure space and $(E,\|\cdot\|)$ be a Banach space with its dual $E^*$ furnished with the dual system $\langle \cdot,\cdot \rangle$ on $E^*\times E$. A function $f:\Omega\to E$ is said to be \textit{strongly measurable} if there exists a sequence of simple (or finitely valued) functions $f_n:\Omega\to E$ such that $\|f(\omega)-f_n(\omega)\|\to 0$ a.e.\ $\omega\in \Omega$; $f$ is said to be \textit{Bochner integrable} if it is strongly measurable and $\int \|f(\omega)\|d\mu<\infty$, where the \textit{Bochner integral} of $f$ over $A\in \Sigma$ is defined by $\int_A fd\mu=\lim_n\int_A f_nd\mu$. By the Pettis measurability theorem (see \cite[Theorem II.1.2]{du77}), $f$ is strongly measurable if and only if it is Borel measurable with respect to the norm topology of $E$ whenever $E$ is separable. Denote by $L^1(\mu,E)$ the space of ($\mu$-equivalence classes of) $E$-\hspace{0pt}valued Bochner integrable functions on $\Omega$ such that $\| f(\cdot) \|\in L^1(\mu)$, normed by $\| f \|_1=\int \| f(\omega) \|d\mu$. A subset $K$ of $L^1(\mu,E)$ is said to be \textit{uniformly integrable} if 
$$
\lim_{\mu(A)\to 0}\sup_{f\in K}\int_A\| f \|d\mu=0. 
$$
 
A function $f:\Omega\to E^*$ is said to be \textit{weakly$^*$ measurable} if for every $x\in E$ the scalar function $\langle f(\cdot),x \rangle:\Omega\to \R$ is measurable. Denote by $L^\infty_{\textit{w}^*}(\mu,E^*)$ the space of weakly$^*$ measurable functions from $\Omega$ to $E^*$ such that $\| f(\cdot) \|_{E^*}\in L^\infty(\mu)$, normed by $\| f \|_\infty=\mathrm{ess.\,sup}\| f(\omega) \|_{E^*}$. The dual space of $L^1(\mu,E)$ is given by $L^\infty_{\textit{w}^*}(\mu,E^*)$ whenever $E$ is separable and the dual system is given by $\langle f,g \rangle=\int\langle f(t),g(t) \rangle d\mu$ with $f\in L^\infty_{\textit{w}^*}(\mu,E^*)$ and $g\in L^1(\mu,E)$; see \cite[Theorem 2.112]{fl07}. Let $F$ be a Banach space and $T:L^1(\mu,E)\to F$ be a linear operator. Then $T$ is norm-\hspace{0pt}to-\hspace{0pt}norm continuous if and only if it is weak-\hspace{0pt}to-\hspace{0pt}weak continuous; see \cite[Theorem 2.5.11]{me98}. In particular, the \textit{integration operator} $T:L^1(\mu,E)\to E$ defined by $Tf=\int fd\mu$ is a norm-\hspace{0pt}to-\hspace{0pt}norm continuous linear operator.  

A set-\hspace{0pt}valued mapping $\Gamma$ from $\Omega$ to the family of nonempty subsets of $E$ is called a \textit{multifunction}. A multifunction $\Gamma:\Omega\twoheadrightarrow E$ is said to be \textit{measurable} if the set $\{\omega\in \Omega\mid \Gamma(\omega)\cap U\ne \emptyset \}$ is in $\Sigma$ for every open subset $U$ of $E$; it is said to be \textit{graph measurable} if the set $\mathrm{gph}\,\Gamma:=\{ (\omega,x)\in \Omega\times E\mid x\in \Gamma(\omega) \}$ belongs to $\Sigma\otimes \mathrm{Borel}(E,\| \cdot \|)$, where $\mathrm{Borel}(E,\| \cdot \|)$ is the Borel $\sigma$-\hspace{0pt}algebra of $(E,\| \cdot \|)$ generated by the norm topology. If $E$ is separable, then $\mathrm{Borel}(E,\| \cdot \|)$ coincides with the Borel $\sigma$-\hspace{0pt}algebra $\mathrm{Borel}(E,\mathit{w})$ of $E$ generated by the weak topology; see \cite[Part I, Chap.\,II, Corollary 2]{sc73}. It is well-\hspace{0pt}known that for closed-\hspace{0pt}valued multifunctions, measurability and graph measurability coincide whenever $E$ is separable; see \cite[Theorem III.30]{cv77}. 

A function $f:\Omega\to E$ is called a \textit{selector} of $\Gamma$ if $f(\omega)\in \Gamma(\omega)$ a.e.\ $\omega\in \Omega$. If $E$ is separable, then by the Aumann measurable selection theorem, a multifunction $\Gamma$ with measurable graph admits a measurable selector (see \cite[Theorem III.22]{cv77}) and it is also strongly measurable. A multifunction $\Gamma:\Omega\twoheadrightarrow E$ is said to be \textit{integrably bounded} if there exists $\varphi\in L^1(\mu)$ such that $\| x \|\le \varphi(\omega)$ for every $x\in \Gamma(\omega)$ a.e.\ $\omega\in \Omega$. If $\Gamma$ is graph measurable and integrably bounded, then it admits a Bochner integrable selector whenever $E$ is separable. Denote by $\mathcal{S}^1_\Gamma$ the set of Bochner integrable selectors of $\Gamma$. A measurable multifunction $\Gamma$ with closed values is integrably bounded if and only if $\mathcal{S}^1_\Gamma$ is bounded in $L^1(\mu,E)$ whenever $E$ is separable; see \cite[Theorem 3.2]{hu77}. The Bochner integral of $\Gamma$ is defined by $\int\Gamma d\mu:=\{ \int fd\mu \mid f\in \mathcal{S}^1_\Gamma \}$. 

The following result provides a sufficient condition for the weak compactness in $L^1(\mu,E)$ that is easy to check. 

\begin{thm}[\citet{drs93}]
\label{thm2}
If $K$ is a bounded and uniformly integrable subset of $L^1(\mu,E)$ such that there exists a multifunction $\Gamma:\Omega\twoheadrightarrow E$ with relatively weakly compact values satisfying $f(\omega)\in \Gamma(\omega)$ for every $f\in K$ and $\omega\in \Omega$, then $K$ is relatively weakly compact in $L^1(\mu,E)$. 
\end{thm}

\noindent
Since $\mathcal{S}^1_\Gamma$ is uniformly integrable in $L^1(\mu,E)$ whenever $\Gamma$ is integrably bounded, it follows from Theorem \ref{thm2} that if $\Gamma$ is integrably bounded with relatively weakly compact values, then $\mathcal{S}^1_\Gamma$ and $\int\Gamma d\mu$ are relatively weakly compact respectively in $L^1(\mu,E)$ and $E$.

\subsection{Lyapunov Convexity Theorem in Banach Spaces}
A finite measure space $(\Omega,\Sigma,\mu)$ is said to be \textit{essentially countably generated} if its $\sigma$-\hspace{0pt}algebra can be generated by a countable number of subsets together with the null sets; $(\Omega,\Sigma,\mu)$ is said to be \textit{essentially uncountably generated} whenever it is not essentially countably generated. Let $\Sigma_S=\{ A\cap S\mid A\in \Sigma \}$ be the $\sigma$-\hspace{0pt}algebra restricted to $S\in \Sigma$. Denote by $L^1_S(\mu)$ the space of $\mu$-integrable functions on the measurable space $(S,\Sigma_S)$ whose elements are restrictions of functions in $L^1(\mu)$ to $S$. An equivalence relation $\sim$ on $\Sigma$ is given by $A\sim B \Leftrightarrow \mu(A\triangle B)=0$, where $A\triangle B$ is the symmetric difference of $A$ and $B$ in $\Sigma$. The collection of equivalence classes is denoted by $\Sigma(\mu)=\Sigma/\sim$ and its generic element $\widehat{A}$ is the equivalence class of $A\in \Sigma$. We define the metric $\rho$ on $\Sigma(\mu)$ by $\rho(\widehat{A},\widehat{B})=\mu(A\triangle B)$. Then $(\Sigma(\mu),\rho)$ is a complete metric space (see \cite[Lemma 13.13]{ab06}) and $(\Sigma(\mu),\rho)$ is separable if and only if $L^1(\mu)$ is separable; see \cite[Lemma 13.14]{ab06}. The \textit{density} of $(\Sigma(\mu),\rho)$ is the smallest cardinal number of the form $|\U|$, where $\U$ is a dense subset of $\Sigma(\mu)$. 

\begin{dfn}
A finite measure space $(\Omega,\Sigma,\mu)$ is \textit{saturated} if $L^1_S(\mu)$ is nonseparable for every $S\in \Sigma$ with $\mu(S)>0$. 
\end{dfn}

\noindent
The saturation of finite measure spaces is also synonymous with the uncountability of the density of $\Sigma_S(\mu)$ for every $S\in \Sigma$ with $\mu(S)>0$; see \cite[331Y(e) and 365X(p)]{fr12}. Saturation implies nonatomicity; in particular, a finite measure space $(\Omega,\Sigma,\mu)$ is nonatomic if and only if the density of $\Sigma_S(\mu)$ is greater than or equal to $\aleph_0$ for every $S\in \Sigma$ with $\mu(S)>0$. Several equivalent definitions for saturation are known; see \cite{fk02,fr12,hk84,ks09}. One of the simple characterizations of the saturation property is as follows. A finite measure space $(\Omega,\Sigma,\mu)$ is saturated if and only if $(S,\Sigma_S,\mu)$ is essentially uncountably generated for every $S\in \Sigma$ with $\mu(S)>0$. An germinal notion of saturation already appeared in \cite{ka44,ma42}. 

For our purpose, the power of saturation is exemplified in the Lyapunov convexity theorem in infinite dimensions. 

\begin{prop}[\citet{ks13}]
\label{lyp}
Let $(\Omega,\Sigma,\mu)$ be a finite measure space and $E$ be a separable Banach space. If $(\Omega,\Sigma,\mu)$ is saturated, then for every $\mu$-continuous vector measure $m:\Sigma\to E$, the range $m(\Sigma)$ is weakly compact and convex. Conversely, if every $\mu$-continuous vector measure $m:\Sigma\to E$ has the weakly compact convex range, then $(\Omega,\Sigma,\mu)$ is saturated whenever $E$ is an infinite-\hspace{0pt}dimensional. 
\end{prop}

\section{Decomposability and Convexity}
\subsection{Decomposability under Nonatomicity}
In what follows, we always assume that a finite measure space $(\Omega,\Sigma,\mu)$ is complete and $E$ is a separable Banach space. Denote by $\overline{C}$ the norm closure of a subset $C$ of $E$. 

\begin{dfn}
A subset $K$ of $L^1(\mu,E)$ is \textit{decomposable} if $\chi_Af+(1-\chi_A)g\in K$ for every $f,g\in K$ and $A\in \Sigma$. 
\end{dfn}

\noindent
It is evident that decomposability is a weaker notion than convexity in $L^1(\mu,E)$. Nevertheless, whenever $(\Omega,\Sigma,\mu)$ is nonatomic, for a weakly closed subset of $L^1(\mu,E)$, these two notions coincide; see \cite[Theorem 2.3.17]{hp97} for the following result. 

\begin{thm}
\label{thm3}
Let $(\Omega,\Sigma,\mu)$ be a nonatomic finite measure space. Then a weakly closed subset $K$ of $L^1(\mu,E)$ is decomposable if and only if $K$ is convex. 
\end{thm}

\noindent
Moreover, decomposable sets are represented by the family of Bochner integrable selectors of a measurable multifunction with closed values; specifically, see \cite[Theorem 3.1]{hu77} for the following result. 

\begin{lem}
\label{lem1}
Let $K\subset L^1(\mu,E)$ be a nonempty closed subset of $L^1(\mu,E)$. Then $K$ is decomposable if and only if there exists a unique measurable multifunction $\Phi:\Omega\twoheadrightarrow E$ with closed values such that $K=\mathcal{S}^1_\Phi$. 
\end{lem}

Under nonatomicity,  we have the following convexity result. 

\begin{thm}
\label{thm4}
Let $(\Omega,\Sigma,\mu)$ be a nonatomic finite measure space. If $K$ is nonempty decomposable subset of $L^1(\mu,E)$, then the set $\overline{\{ \int fd\mu\mid f\in K \}}$ is convex. If, furthermore, $K$ is bounded and weakly closed such that there exists a multifunction $\Gamma:\Omega\twoheadrightarrow E$ with relatively weakly compact values satisfying $f(\omega)\in \Gamma(\omega)$ for every $f\in K$ and $\omega\in \Omega$, then the set $\{ \int fd\mu\mid f\in K \}$ is weakly compact and convex. 
\end{thm}

\begin{proof}
The convexity of $\overline{\{ \int fd\mu\mid f\in K \}}$ for every nonempty decomposable set $K$ follows from \cite[Corollary 3.16]{hp97}. Suppose further that $K$ is bounded and weakly closed. By Lemma \ref{lem1}, there exists a unique measurable multifunction $\Phi:\Omega\twoheadrightarrow E$ with closed values such that $K=\mathcal{S}^1_\Phi$. Since $\mathcal{S}^1_\Phi$ is bounded, $\Phi$ is integrably bounded, and hence, $\mathcal{S}^1_\Phi$ is uniformly integrable as noted in Subsection \ref{subsec1}. Therefore, $K$ is weakly compact by Theorem \ref{thm1}. Since $K$ is also convex by Theorem \ref{thm3}, $\{ \int fd\mu\mid f\in K \}$ is weakly compact and convex because of the weak-to-weak continuity of the integration operator from $L^1(\mu,E)$ to $E$ given by $f\mapsto \int fd\mu$. 
\end{proof}

\begin{exmp}
\label{exmp1}
The closure operation cannot be removed from Theorem \ref{thm4}. Suppose that $(\Omega,\Sigma,\mu)$ is a nonatomic finite measure space that is essentially countably generated. (By the classical isomorphism theorem, such a measure space is isomorphic to the Lebesgue measure space of a real interval.) If $E$ is an infinite-\hspace{0pt}dimensional separable Banach space, then there exists $f\in L^1(\mu,E)$ such that the set $R_f:=\{\int_Afd\mu\mid A\in \Sigma \}$ is not convex in $E$; see \cite[Lemma 4]{po08} or \cite[Remark 1(2)]{sy08}. Let $K=\{ \chi_Af\mid A\in \Sigma \}$. Then $K$ is a decomposable subset of $L^1(\mu,E)$ such that $\{\int gd\mu\mid g\in K \}=R_f$ is not convex. This observation also demonstrates a failure of the Lyapunov convexity theorem in infinite dimensions. Define the $\mu$-continuous vector measure $m_f:\Sigma\to E$ by $m_f(A)=\int_Afd\mu$. We then have the range $m_f(\Sigma)=R_f$ is not convex. These counterexamples stem from the fact that $K$ is not weakly closed.  
\end{exmp}

\begin{rem}
\label{rem1}
The reason that the closure operation in Theorem \ref{thm4} is inevitable lies in the fact that Uhl's approximate Lyapunov convexity theorem is employed for its proof: \textit{The norm closure $\overline{m(\Sigma)}$ of the range of a vector measure $m:\Sigma\to E$ of the form $m(A)=\int_Afd\mu$ with $f\in L^1(\mu,E)$ and $A\in \Sigma$ is norm compact and convex whenever $(\Omega,\Sigma,\mu)$ is nonatomic;} see \cite{uh69}. When $E=\R^n$, the closure operation is unnecessary and Theorem \ref{thm4} reduces to Theorem \ref{thm1}.  
\end{rem}

\subsection{Decomposability under Saturation}
We are now in a position to state the main result of the paper.  

\begin{thm}
\label{thm5}
If $(\Omega,\Sigma,\mu)$ is saturated, then for every nonempty decomposable subset $K$ of $L^1(\mu,E)$, every separable Banach space $F$, and every continuous linear operator $T:L^1(\mu,E)\to F$, the set $T(K)$ is convex. Conversely, if for every nonempty decomposable subset $K$ of $L^1(\mu,E)$, every separable Banach space $F$, and every continuous linear operator $T:L^1(\mu,E)\to F$, the set $T(K)$ is convex, then $(\Omega,\Sigma,\mu)$ is saturated whenever $E$ is infinite dimensional. 
\end{thm}

\begin{proof}
Suppose that $(\Omega,\Sigma,\mu)$ is saturated. Take any $x,y\in T(K)$ and $\alpha\in [0,1]$. For the convexity of $T(K)$, it suffices to show that $\alpha x+(1-\alpha)y\in T(K)$. Toward this end, let $f,g\in K$ be such that $x=Tf$ and $y=Tg$ and define the vector measure $m:\Sigma \to F$ by $m(A)=T(\chi_A(f-g))$. To demonstrate the countable additivity of $m$, let $\{ A_n \}_{n=1}^\infty$ be a pairwise disjoint sequence in $\Sigma$. If $S_k=\bigcup_{n=k+1}^\infty A_n$, then by the linearity of $T$, we have $m(\bigcup_{n=1}^\infty A_n)=\sum_{n=1}^kT(\chi_{A_n}(f-g))+T(\chi_{S_k}(f-g))$. Thus, $\|m(\bigcup_{n=1}^\infty A_n)-\sum_{n=1}^km(A_n) \|=\| T(\chi_{S_k}(f-g))\|$. Since $\chi_{S_k}(f-g)\to 0$ in $L^1(\mu,E)$ as $k\to \infty$, the continuity of $T$ yields $m(\bigcup_{n=1}^\infty A_n)=\sum_{n=1}^\infty m(A_n)$. Hence, $m$ is countably additive. Since $m$ is absolutely continuous with respect to the saturated measure $\mu$, it follows from Proposition \ref{lyp} that $m(\Sigma)$ is weakly compact and convex. Thus, there exists $A\in \Sigma$ such that $m(A)=\alpha m(\Omega)$, and hence, $T(\chi_A(f-g))=\alpha T(f-g)$. This means that
\begin{align*}
\alpha x+(1-\alpha)y=\alpha T(f-g)+T(g)
& =T(\chi_A(f-g))+T(g) \\
& =T(\chi_Af+(1-\chi_A)g)\in T(K),
\end{align*}
where the last inclusion follows from the decomposability of $K$.

To show the converse implication, assume that $(\Omega,\Sigma,\mu)$ is not saturated. Since $E$ is an infinite-\hspace{0pt}dimensional separable Banach space, there exists $f\in L^1(\mu,E)$ such that the set $\{\int_Afd\mu\mid A\in \Sigma \}$ is not convex in $E$; see \cite[Lemma 4]{po08} or \cite[Remark 1(2)]{sy08}. As in Example \ref{exmp1}, define the decomposable set by $K:=\{ \chi_Af\mid A\in \Sigma \}$. Then the integration operator $T:L^1(\mu,E)\to E$ defined by $Tf=\int fd\mu$ is such that the image $T(K)=\{\int_Afd\mu\mid A\in \Sigma \}$ is not convex. 
\end{proof}

\begin{cor}
\label{cor1}
Let $E$ be an infinite-dimensional separable Banach space. Then the following conditions are equivalent:
\begin{enumerate}[\rm (i)]
\item $(\Omega,\Sigma,\mu)$ is saturated. 
\item For every $\mu$-continuous vector measure $m:\Sigma\to E$, the range $m(\Sigma)$ is weakly compact and convex.  
\item For every nonempty decomposable subset $K$ of $L^1(\mu,E)$, the set $\{ \int fd\mu\mid f\in K \}$ is convex.
\end{enumerate}  
In particular, the implication (i) $\Rightarrow$ (ii), (iii) is true for every separable Banach space. 
\end{cor}

\begin{proof}
(i) $\Leftrightarrow$ (ii): Immediate from Proposition \ref{lyp}; (i) $\Rightarrow$ (iii): Letting $F=E$ and $T$ to be the integration operator in Theorem \ref{thm5} yields the result; (iii) $\Rightarrow$ (i): Immediate from the proof of Theorem \ref{thm5}. The last part of the assertion is a just repetition of Proposition \ref{lyp} and Theorem \ref{thm5}. 
\end{proof}

\noindent
Corollary \ref{cor1} implies that a ``strengthened'' version of Theorem \ref{thm4} without the closure operation under saturation.

\end{document}